\documentclass[11pt,leqno]{article}
\usepackage[utf8]{inputenc}
\usepackage{amsmath}
\usepackage{amsfonts}
\usepackage{amssymb}
\usepackage{graphicx}
\usepackage{mathrsfs}
\usepackage{mathrsfs}
\usepackage{upref,amsthm,amsxtra,exscale}
\usepackage[nosort]{cite}
\usepackage[colorlinks=true,urlcolor=blue,
citecolor=red,linkcolor=blue,linktocpage,pdfpagelabels,
bookmarksnumbered,bookmarksopen]{hyperref}
\allowdisplaybreaks[4]
\usepackage{fullpage}

\newtheorem{theorem}{Theorem}[section]
\newtheorem{corollary}[theorem]{Corollary}
\newtheorem{remark}[theorem]{Remark}

\newtheorem{lemma}[theorem]{Lemma}
\newtheorem{proposition}[theorem]{Proposition}

\newtheorem{example}[theorem]{Example}

\numberwithin{equation}{section}

\def\r{\mathbb{R}}
\def\rn{\mathbb{R}^N}

\def\n{\mathbb{N}}

\def\rh{\rightharpoonup}

\def\irn{\int_{\r^N}}

\def\wt{\widetilde}
\def\wh{\widehat}
\def\ol{\overline}
\def\cC{\mathcal{C}}

\def\cH{\mathcal{H}}

\def\cJ{\mathcal{J}}

\def\cO{\mathcal{O}}

\def\cU{\mathcal{U}}
\def\cV{\mathcal{V}}
\def\cW{\mathcal{W}}
\def\vr{\varrho}
\def\Hr{H^1_\mathrm{rad}(\rn)}
\def\d12{\mathcal{D}^{1,2}}

\title{Normalized solutions to a non-variational Schrödinger system}
\author{Mónica Clapp\footnote{M. Clapp was partially supported by CONACYT grant A1-S-10457 (Mexico).}\; and Andrzej Szulkin}
\date{}

\begin{document}

\maketitle

\begin{abstract}
\noindent We establish the existence of positive normalized (in the $L^2$ sense) solutions to non-variational weakly coupled elliptic systems of $\ell$ equations. We consider couplings of both cooperative and competitive type. We show the problem can be formulated as an operator equation on the product of $\ell$ $L^2$-spheres and apply a degree-theoretical argument on this product to obtain existence.

\medskip


\medskip

\end{abstract}

\section{The problem} \label{intro}

Let $\bar\kappa=(\kappa_1,\ldots,\kappa_\ell)$ and $\bar u=(u_1,\ldots,u_\ell)$. We look for solutions $(\bar\kappa,\bar u)$ to the system
\begin{equation} \label{eq:mainsystem}
\begin{cases}
-\Delta u_i + \kappa_i u_i = \mu_i u_i^p + \sum\limits_{\substack{j=1 \\ j\neq i}}^\ell\lambda_{ij}u_i^{\alpha_{ij}}u_j^{\beta_{ij}}, \\
\kappa_i\in\r, \quad u_i\in H^1_\mathrm{rad}(\rn),\quad u_i > 0,\quad i=1,\ldots,\ell,
\end{cases}
\end{equation}
which satisfy the constraints
\begin{equation} \label{eq:constraint}
|u_i|_2^2:=\irn u_i^2=r_i^2,\qquad i=1,\ldots,\ell,
\end{equation}
for prescribed \ $r_i>0$. Throughout the paper the following conditions will be assumed:
\begin{itemize}
\item[$(H)$] \ $2\le N\leq 4$, \ $1+\frac4N<p<\frac{N+2}{N-2}$ ($\frac{N+2}{N-2} := \infty$ if $N=2$), \ $\mu_i>0$ and $\alpha_{ij}\ge 1,\ \beta_{ij}>0$.
\end{itemize}
 Note that, multiplying the $i$-th equation by $u_i$ and integrating,
\[
\kappa_i r_i^2 = \irn\Big( \mu_iu_i^{p+1} + \sum\limits_{\substack{j=1 \\ j\neq i}}^\ell\lambda_{ij}u_i^{\alpha_{ij}+1}u_j^{\beta_{ij}}\Big) - \irn|\nabla u_i|^2. 
\]
The system \eqref{eq:mainsystem} subject to the constraints \eqref{eq:constraint} is non-variational except for a very special choice of $\lambda_{ij}$ and $\alpha_{ij}, \beta_{ij}$. In a variational setting existence of normalized solutions, both for a single equation and for systems, has been studied under a variety of assumptions in a large number of papers. See e.g. \cite{bjj, bjjs, bs, bzz, bm, i, jj, jjlu, lz, ms, ntv} and the references there.  Such systems have applications in different models in physics, e.g. they arise when studying mixtures of Bose-Einstein condensates or propagation of wave packets in nonlinear optics. The $L^2$-norms respectively represent the number of particles and power supply. The parameters $\mu_i>0$  represent the attractive interaction between particles of the same kind and $\lambda_{ij}$ the interaction between particles of different type. It can be attractive ($\lambda_{ij}>0$) or repulsive ($\lambda_{ij}<0$). The a priori unknown parameters $\kappa_i$ come from the time-dependent Schr\"odinger equations
\[
i\frac{\partial\Phi_i}{\partial t} -\Delta \Phi_i = \mu_i |\Phi_i|^{p-1}\Phi_i + \sum\limits_{\substack{j=1 \\ j\neq i}}^\ell\lambda_{ij}|\Phi_i|^{\alpha_{ij}-1}|\Phi_j|^{\beta_{ij}}\Phi_i
\]
after making the standing wave ansatz $\Phi_i(x,t) =  e^{-i\kappa_it}u_i(x)$. Note that the $L^2$-norm is conserved under the time-evolution of $\Phi_i$, i.e., $|\Phi_i(\cdot,t)|_2 = |u_i|_2 = r_i$ for all $t\in\r$.   More information about the physical background, together with references to relevant literature in physics, may be found in the papers quoted above.

Our main results in this paper are the following two theorems:

\begin{theorem} \label{mainthm}
In addition to $(H)$, suppose that $\lambda_{ij}>0,\ \alpha_{ij} > 1+\frac4N$ and $\alpha_{ij}+\beta_{ij}<p$ ($i,j=1,\ldots,\ell$, $j\neq i$). Then the system \eqref{eq:mainsystem} subject to the constraints \eqref{eq:constraint} has a solution. If $N=2$, or $N=3$ and $p\le \frac N{N-2}$, then the inequality $\alpha_{ij}+\beta_{ij}<p$ may be replaced by $\alpha_{ij}+\beta_{ij}\le p$.
\end{theorem}

\begin{theorem} \label{mainthm2}
In addition to $(H)$, suppose that $N=2$ or $3$, $\lambda_{ij}<0$ and $1+\frac4{N-1} < \alpha_{ij}+\beta_{ij} < p$ ($i,j=1,\ldots,\ell$, $j\neq i$). Then the system \eqref{eq:mainsystem} subject to the constraints \eqref{eq:constraint} has a solution.
\end{theorem}

Concerning the assumptions, we always have $p>1+\frac4N$, i.e. the first term on the right-hand side of the equation in \eqref{eq:mainsystem} is mass-supercritical (but Sobolev-subcritical). We do not know if the  assumption $\alpha_{ij} > 1+\frac4N$ in Theorem \ref{mainthm} can be weakened but it cannot be simply replaced by $\alpha_{ij}\ge 1$, see \cite[Proposition 2.2]{bzz} where a nonexistence result is shown for a (variational) system of 2 equations under suitable assumptions on the parameters (in particular, $\alpha_{ij}=1$ for  $i=1,2$ there). We believe the assumption $\alpha_{ij}+\beta_{ij} > 1+\frac4{N-1}$ in Theorem \ref{mainthm2} can be weakened to the more natural $\alpha_{ij}+\beta_{ij} > 1+\frac4{N}$. Finally, it seems difficult to replace the assumption $\alpha_{ij}+\beta_{ij} < p$ by $\alpha_{ij}+\beta_{ij} \le p$ if $\lambda_{ij}<0$ because there may be no a priori $L^\infty$-bound for positive solutions as suggested by the results in \cite{dww}. Existence of such bound is crucial for our arguments. 

\medskip

Our proofs employ a homotopy argument and for this purpose we need to consider the parametri\-zed system
\begin{equation} \label{eq:system}
\begin{cases}
-\Delta u_i + \kappa_i u_i = \mu_i (u_i^+)^p + t\sum\limits_{\substack{j=1 \\ j\neq i}}^\ell\lambda_{ij}(u_i^+)^{\alpha_{ij}}(u_j^+)^{\beta_{ij}}=:f_i^t(\bar u), \\
\kappa_i\in\r,\quad 0\le t\le 1, \quad u_i\in H^1_\mathrm{rad}(\rn),\quad u_i\neq 0,\quad i=1,\ldots,\ell,
\end{cases}
\end{equation}
subject to the same constraints \eqref{eq:constraint}, where $\bar u=(u_1,\ldots,u_\ell)$ and $u_i^+ := \max\{u_i,0\}$, $u_i^- := \min\{u_i,0\}$. As above, we have $\kappa_i=\kappa_i^t$ where
\begin{equation} \label{kappa}
\kappa_i^t = \kappa_i^t(\bar u) = r_i^{-2}\left(\irn f_i^t(\bar u)u_i - \irn|\nabla u_i|^2\right).
\end{equation}
 Let
\begin{equation} \label{eq:ui}
 \mathscr{U}^t := \{\bar u=(u_1,\ldots,u_\ell): u_i\in H^1_\mathrm{rad}(\rn), \ |u_i|_2 = r_i \text{ and } \kappa_i^t(\bar u)>0 \text{ for all } i\}.
\end{equation}
We shall only be interested in solutions of \eqref{eq:system} which are contained in $\mathscr{U}^t$.

The paper is organized as follows. In Section \ref{apb} we show that solutions $(\bar\kappa,\bar u)$ of \eqref{eq:system} in $\mathscr{U}^1$ solve the problem \eqref{eq:mainsystem}-\eqref{eq:constraint}. We also show that there exist a priori bounds necessary for our homotopy argument and that a sequence of Galerkin-type approximations converges to a solution of \eqref{eq:system}. Section \ref{comput} deals with the case $t=0$ (uncoupled system). In particular, we compute the Morse index of the functional which corresponds to the uncoupled system and is defined on the torus $\mathscr{T}$ (see the notation below). In Section \ref{pfmain} Theorems \ref{mainthm} and \ref{mainthm2} are proved. In Appendix \ref{sec:hessian} we discuss the Hessian of the functional appearing Section \ref{comput}. 

\medskip

\noindent \textbf{Notation.}
 \ $|\,\cdot\,|_q$ will denote the $L^q$-norm for $1\le q\le\infty$, $\Hr := \{u\in H^1(\rn): u(x)=u(|x|) \text{ for all }x\in\rn\}$, \ $\langle\,\cdot\,,\,\cdot\,\rangle$ and $\|\,\cdot\,\|$ stand for the usual inner product and norm in $H^1(\rn)$, and
\[
\cH := [\Hr]^\ell, \qquad \mathscr S_i := \{u_i\in H^1_\mathrm{rad}(\rn): |u_i|_2=r_i\}, \qquad \mathscr{T} := \mathscr S_1\times\cdots\times \mathscr S_\ell.
\]

\section{A priori bounds} \label{apb}

\begin{lemma} \label{lem:positive}
Suppose $\lambda_{ij}>0$ for all $i,j=1,\ldots,\ell$, $j\neq i$.
If $(\bar\kappa,\bar u)$ is a solution to the system \eqref{eq:system} for some $t\in[0,1]$, then $\kappa_i\ne 0$ for any $i$. If furthermore $\bar u\in\mathscr{U}^t$, then $u_i=u_i^+$ for all $i=1,\ldots,\ell$. In particular, if $(\bar\kappa,\bar u)$ is a solution to \eqref{eq:system} for $t=1$ with $\bar u\in\mathscr{U}^1$, then it is also a solution to \eqref{eq:mainsystem}-\eqref{eq:constraint}.
\end{lemma}

\begin{proof}
If $(\bar\kappa,\bar u)$ solves the system \eqref{eq:system} and $\kappa_i=0$, then $-\Delta u_i^-=0$. Hence, $u_i=u_i^+$.  Then, as $u_i\in H^1_\mathrm{rad}(\rn)\subset L^q(\rn)$ for all $q\in[2,\frac{2N}{N-2}]$, $N\leq 4$ and $-\Delta u_i\geq 0$, it follows from \cite[Lemma A.2]{i} that $u_i\equiv 0$. This is a contradiction. 

If $(\bar\kappa,\bar u)$ solves the system \eqref{eq:system} and $\bar u\in\mathscr{U}^t$, then $-\Delta u_i^-+\kappa_iu_i^-=0$ with $\kappa_i>0$. Hence $u_i=u_i^+$ as claimed.

The last conclusion is obvious (that $u_i$ is strictly positive is a consequence of the maximum principle).
\end{proof}

\begin{lemma} \label{lem:positive2}
Suppose $\lambda_{ij}<0$, $N=2$ or $3$ and $1+\frac4{N-1} < \alpha_{ij}+\beta_{ij} < p$ for all $i,j=1,\ldots,\ell$, $j\neq i$.
If $(\bar\kappa,\bar u)$ is a solution to the system \eqref{eq:system} for some $t\in[0,1]$, then $\kappa_i\ne 0$ for any $i$. If furthermore $\bar u\in\mathscr{U}^t$, then $u_i=u_i^+$ for all $i=1,\ldots,\ell$. In particular, if $(\bar\kappa,\bar u)$ is a solution to \eqref{eq:system} for $t=1$ with $\bar u\in\mathscr{U}^1$, then it is also a solution to \eqref{eq:mainsystem}-\eqref{eq:constraint}.
\end{lemma}

\begin{proof}
The argument excluding $\kappa_i=0$ is different here while the proof that $u_i=u_i^+$ is the same as above. 

We adapt the argument of \cite[Lemma 3.12]{bs}. Suppose $(\bar\kappa,\bar u)$ solves the system \eqref{eq:system} and $\kappa_i=0$. Then $u_i\ge 0$ and 
\[
-\Delta u_i = \mu_iu_i^p + t\sum\limits_{\substack{j=1 \\ j\neq i}}^\ell\lambda_{ij}u_i^{\alpha_{ij}}(u_j^+)^{\beta_{ij}}. 
\]
Since $u_i,u_j\in \Hr$, it follows from \cite[Radial Lemma 1]{st} that
\[
0\le u_i,u_j^+\le d_1|x|^{\frac{1-N}2}
\]
for some constant $d_1>0$. Put
\[
c(x) := -t\sum\limits_{\substack{j=1 \\ j\neq i}}^\ell\lambda_{ij}u_i^{\alpha_{ij}-1}(u_j^+)^{\beta_{ij}}.
\]
Then
\[
0 \le c(x) \le d_2|x|^{-\frac{\gamma(N-1)}2} \quad \text{where } \gamma := \min_{j\ne i}(\alpha_{ij}+\beta_{ij}-1).
\]
As $-\Delta u_i + c(x)u_i\ge 0$, the maximum principle implies $u_i>0$ in $\rn$.
Let $r:=|x|$ and $v(x) = \bar v(r) := r^{-\delta}$. Since $\Delta v = \bar v''(r)+ \frac{N-1}r \bar v'(r)$, we have
\[
-\Delta v + c(x)v \le \delta(N-2-\delta)r^{-\delta-2} + d_3r^{-\frac{\gamma(N-1)}2-\delta}
\] 
for some $d_3>0$.  Choose $\delta\in (N-2, N/2]$ (it is here the assumption $N<4$ enters). As $\gamma > \frac4{N-1}$, it follows that $\frac{\gamma(N-1)}2>2$ and therefore $-\Delta v + c(x)v < 0$ provided $|x|>r_0$ and $r_0$ is large enough. Let $w(x) := u_i(x)-a_0v(x)$ where $a_0>0$ is such that $w(x)\ge 0$ for $|x|=r_0$. Then $w$ satisfies
\[
\left\{
\begin{tabular}{ll}
$-\Delta w + c(x) w \ge 0$, & $|x|>r_0$ \\
$w(x)\ge 0$, & $|x|=r_0$.
\end{tabular}
\right.
\]
By the maximum principle \cite[Theorem 1.28]{dp}, $w>0$, i.e., $u_i>a_0v$ for $|x|>r_0$. It follows that
\[
\int_{|x|>r_0}u_i^2 \ge a_0^2\int_{|x|>r_0}v^2 = a\int_{r_0}^\infty r^{N-1-2\delta} = \infty  
\]
because $\delta\le N/2$. As $u_i\in L^2(\rn)$, this is a contradiction.
 
Using the maximum principle again, we see that if $(\bar\kappa,\bar u)$ solves \eqref{eq:mainsystem}-\eqref{eq:constraint} and $u_i\in\mathscr{U}^1$, then $u_i>0$ in $\rn$.
\end{proof}

For $i=1,\ldots,\ell$ \ set
\begin{align*}
\mathcal Q_i:=&\,\{\kappa_i:(\bar\kappa,\bar u)\text{ solves \eqref{eq:system} and } \bar u\in\mathscr{U}^t \text{ for some }t \in[0,1]\}, \\
\mathscr K_i:=&\,\{u_i:(\bar\kappa,\bar u)\text{ solves \eqref{eq:system} and } \bar u\in\mathscr{U}^t \text{ for some }t \in[0,1]\}.
\end{align*}

From now on, we assume that either the hypotheses of Theorem \ref{mainthm}, or those of Theorem \ref{mainthm2} are satisfied. 

\begin{lemma} \label{lem:infty}
The set $\mathscr K_i$ is uniformly bounded in $L^\infty(\rn)$ for each $1\le i\le\ell$.
\end{lemma}

\begin{proof}
Suppose there exists a sequence of solutions  $(\bar{\kappa}_n,\bar u_n)$  to \eqref{eq:system} for $t=t_n\in[0,1]$ with $\bar u_n\in\mathscr{U}^{t_n}$, such that $|\bar u_n|_\infty\to\infty$. Passing to a subsequence, we may assume $|u_{n,i}|_\infty\to\infty$ and  $|u_{n,i}|_\infty\ge|u_{n,j}|_\infty$ for some $i$ and all $j$. By standard regularity results all $u_{n,j}$ are in $\cC^2(\rn)$. Since $u_{n,i}\geq 0$, $u_{n,i}\neq 0$ and $u_{n,i}(x)\to 0$ as $|x|\to\infty$, \ $u_{n,i}$ attains it maximum at some $x_n\in\rn$. 

For any given $\theta>0$ (to be specified), let $\vr_n>0$ be defined by
\[
\vr_n^\theta\,|u_{n,i}|_\infty = 1.
\]
Clearly, $\vr_n\to 0$. Let
\begin{equation} \label{defvni}
v_{n,j}(y) := \vr_n^\theta u_{n,j}(\vr_ny+x_n), \quad j=1,\ldots, \ell.
\end{equation}
Then, for all $n$ and $j$, $v_{n,j}\in\cC^2(\rn)$,
\begin{equation} \label{vni}
0\le v_{n,j}\le1, \quad v_{n,i}(0)=1 \quad \text{and } v_{n,j}(x)\to 0 \text{ as } |x|\to\infty. 
\end{equation}
It follows from  \eqref{defvni} and Lemma \ref{lem:positive} or Lemma \ref{lem:positive2} that $\kappa_{n,i}>0$, \ $v_{n,j}=v_{n,j}^+$ \ and
\begin{align*}
-\Delta_y v_{n,i}  & = \vr_n^{\theta+2}\Delta_x u_{n,i} = \vr_n^{\theta+2}\Big(-\kappa_{n,i}u_{n,i} + \mu_iu_{n,i}^p+t_n\sum\limits_{j\ne i}\lambda_{ij}u_{n,i}^{\alpha_{ij}}u_{n,j}^{\beta_{ij}}\Big) \\
& = -\vr_n^2\kappa_{n,i} v_{n,i} + \vr_n^{\theta+2-\theta p}\mu_i v_{n,i}^p + t_n\sum\limits_{j\ne i} \vr_n^{\theta+2-\theta(\alpha_{ij}+\beta_{ij})}\lambda_{ij}v_{n,i}^{\alpha_{ij}}v_{n,j}^{\beta_{ij}},
\end{align*}
or equivalently,
\begin{equation} \label{eq:equality}
-\Delta v_{n,i} + \vr_n^2\kappa_{n,i} v_{n,i} = \vr_n^{\Gamma}\mu_i v_{n,i}^p + t_n\sum\limits_{j\ne i} \vr_n^{\gamma_{ij}}\lambda_{ij}v_{n,i}^{\alpha_{ij}}v_{n,j}^{\beta_{ij}}
\end{equation}
where $\Gamma := \theta+2-\theta p$ and $\gamma_{ij} := \theta+2-\theta(\alpha_{ij}+\beta_{ij})$. 

\medskip

Suppose first $\alpha_{ij}+\beta_{ij} < p$ for all $i$ and $j\ne i$. We shall consider three cases and show that none of them can occur.

\noindent\emph{Case 1.} $ \vr_n^2\kappa_{n,i}\to\infty$ up to a subsequence.\\
Choose $\theta = 2/(p-1)$. Then  $\Gamma=0$ and $\gamma_{ij}>0$. So the right-hand side of \eqref{eq:equality} is uniformly bounded while the left-hand side is not because $v_{n,i}(0)=1$ and, as $v_{n,i}$ attains its maximum at $0$, $-\Delta v_{n,i}(0)\ge 0$.

\noindent\emph{Case 2.} $ \vr_n^2\kappa_{n,i}\to 0$ up to a subsequence.\\
The argument is the same as in Case 1 of the proof of Lemma 1.2 in \cite[Appendix A]{cs} and goes back to \cite{gs, dfy}. We repeat it briefly. Let $\theta$ be as above (so $\Gamma=0, \ \gamma_{ij}>0$). Fix $R>0$. By elliptic estimates, $(v_{n,i})$ is bounded in $W^{2,q}(B_R(0))$ for some $q>N$. So passing to a subsequence, $v_{n,i} \to v_i$ weakly in $W^{2,q}(B_R(0))$ and strongly in $\cC^1(B_R(0))$. Since $\vr_n^{\gamma_{ij}}\to 0$, \ $v_i\ge 0$ and satisfies  
\[
-\Delta v_i = \mu_iv_i^p \quad \text{in } B_R(0). 
\]
Let $R_m\to\infty$. For each $m$ we get a solution $v_i^{m}$ of the above equation in $B_{R_m}(0)$. Passing to subsequences of $m$ and applying the diagonal procedure we see that $v_i^{m}\to w_i$, weakly in $W^{2,q}_{loc}(\rn)$ and strongly in $\cC^1_{loc}(\rn)$. So $-\Delta w_i = \mu_iw_i^p$ in $\rn$, $w_i\ge 0$, $w_i(0)=1$ according to \eqref{vni}, and $w_i\in \cC^2(\rn)$ by the Schauder estimates. According to \cite[Theorem 1.2]{gs}, $w_i=0$ which is a contradiction.

\noindent\emph{Case 3.} $\vr_n^2\kappa_{n,i}\to c>0$ up to a subsequence.\\
Choose $\theta< 2/(p-1)$. Then $\Gamma$ and $\gamma_{ij}$ are positive. Since now also $\vr_n^\Gamma\to 0$, using the same argument as in Case 2 we obtain $w_i\ge 0$ such that $w_i(0)=1$ and $-\Delta w_i + cw_i =0$ in $\rn$. Since $c>0$, this is impossible.

As the three cases above exhaust all possibilities, the proof is complete. 

\medskip

Next we assume $\lambda_{ij}>0$, $N=2$ or 3 and $\alpha_{ij}+\beta_{ij} \le p$ for all $i$ and $j\ne i$. If $N=3$, we also assume $p\le \frac N{N-2}$. In Case 1, $\Gamma=0$ and $\gamma_{ij}\ge 0$, and in Case 3, $\Gamma$ and $\gamma_{ij}$ are positive, so the proof is exactly the same. In Case 2 the same argument using \eqref{eq:equality}  gives
\[
-\Delta v_i = \mu_iv_i^p + t\sum\limits_{j\in A_i} \lambda_{ij}v_{i}^{\alpha_{ij}}v_{j}^{\beta_{ij}}
\]
where $A_i := \{j\ne i: \alpha_{ij}+\beta_{ij} = p\}$ ($A_i$ may be empty). Note that here we need not only $v_{n,i}\to v_i$ but also $v_{n,j}\to v_j$ for $j\in A_i$ which we can obtain by using the $j$-th equation. The diagonal procedure now leads to 
\[
-\Delta w_i = \mu_iw_i^p + t\sum\limits_{j\in A_i} \lambda_{ij}w_{i}^{\alpha_{ij}}w_{j}^{\beta_{ij}} \ge \mu_iw_i^p \quad \text{in } \rn
\]
(a similar agument appears in \cite{dww}).
As $w_i\ge 0$, the above inequality can hold only for $w_i=0$ \cite[Theorem 2.1]{mp}. This is impossible because $w_i(0)=1$. Note that the result in \cite{mp} holds for $1<p\le \frac{N}{N-2}$ if $N\ge 3$. So taking into account the assumptions of Theorem \ref{mainthm}, we have $1+\frac4N < p\le \frac{N}{N-2}$ which is satisfied only for $N=3$. If $N=2$, then an easy inspection of the proof in \cite{mp} shows that the result is true for any $p>1$. 
\end{proof}

\begin{corollary} \label{cor:bound}
The set $\mathscr K_i$ is uniformly bounded in $H^1_\mathrm{rad}(\rn)$ and the set $\mathcal Q_i$ is uniformly bounded in $\r$ for each $1\le i\le\ell$.
\end{corollary}

\begin{proof}
If $(\bar\kappa,\bar u)$ solves \eqref{eq:system} and $\bar u\in\mathscr{U}^t$ for some $t \in[0,1]$, then
 $\kappa_i>0$ and $u_i=u_i^+$ by Lemma \ref{lem:positive} or Lemma \ref{lem:positive2}. Multiplying the $i$-th equation in \eqref{eq:system} by $u_i$ and integrating we obtain
\[
\irn|\nabla u_i|^2 + \kappa_i r_i^2 = \irn f_i^t(\bar u)u_i. 
\]
Since $|u_i|_2 = r_i$, it follows from Lemma \ref{lem:infty} that the set $\mathscr K_i$ is bounded in $L^q(\rn)$ for any $2\le q\le \infty$. Therefore, the right-hand side above is uniformly bounded and, as a consequence, so are $|\nabla u_i|_2$ and $\kappa_i$. Hence the conclusion. 
\end{proof}

\begin{proposition} \label{lem:H_bounds}
There are positive constants $a_1,a_2,A_1,A_2$ such that
\begin{align*}
0<a_1\leq\kappa_i\leq A_1&\qquad\text{for all \ }\kappa_i\in\mathcal Q_i, \ i=1,\ldots,\ell, \\
0<a_2\leq |\nabla u_i|_2^2:=\irn|\nabla u_i|^2\leq A_2&\qquad\text{for all \ }u_i\in\mathscr K_i, \ i=1,\ldots,\ell.
\end{align*}
\end{proposition}

\begin{proof}
It follows from Corollary \ref{cor:bound} that $A_1, A_2$ as above exist, so we must show that so do $a_1,a_2$. Suppose first $\lambda_{ij}>0$ for all $i,j=1,\ldots,\ell$, $j\neq i$.
If $(\bar\kappa,\bar u)$ solves \eqref{eq:system} and $\bar u\in\mathscr{U}^t$ for some $t$, then 
\begin{align} \label{eq:H_bounds}
&\irn|\nabla u_i|^2+\kappa_i r_i^2=\mu_i\irn|u_i|^{p+1}+t\sum\limits_{\substack{j=1 \\ j\neq i}}^\ell\lambda_{ij}\irn|u_i|^{\alpha_{ij}+1}|u_j|^{\beta_{ij}}\\ 
&\qquad\leq \mu_i|u_i|_{p+1}^{p+1} + C_1\sum\limits_{\substack{j=1 \\ j\neq i}}^\ell |u_i|^{\alpha_{ij}+1}_{\alpha_{ij}+\beta_{ij}+1} |u_j|^{\beta_{ij}}_{\alpha_{ij}+\beta_{ij}+1} \le C_2\Big(|u_i|_{p+1}^{p+1} + \sum\limits_{\substack{j=1 \\ j\neq i}}^\ell|u_i|^{\alpha_{ij}+1}_{\alpha_{ij}+\beta_{ij}+1}\Big) \nonumber 
\end{align}
because according to Corollary \ref{cor:bound}, $u_j$ are uniformly bounded in $\Hr\subset L^{\alpha_{ij}+\beta_{ij}+1}(\rn)$ for every $i,j=1,\ldots,\ell$, $i\ne j$. Since $\kappa_i\ge 0$, \eqref{eq:H_bounds} and  the Gagliardo-Nirenberg inequality yield
\begin{equation}\label{eq:gn}
|\nabla u_i|_2^2\leq C_3\Big(|\nabla u_i|_2^\frac{N(p-1)}{2}+\sum\limits_{\substack{j=1 \\ j\neq i}}^\ell|\nabla u_i|_2^{N(\alpha_{ij}+1)(\frac12-\frac1{\alpha_{ij}+\beta_{ij}+1})}\,\Big).
\end{equation}
As $p,\alpha_{ij}>1+\frac4N$ and  
\[
N(\alpha_{ij}+1)(\tfrac12-\tfrac1{\alpha_{ij}+\beta_{ij}+1}) > N(\alpha_{ij}+1)(\tfrac12-\tfrac1{\alpha_{ij}+1})
\]
(recall $\beta_{ij}>0$), all exponents on the right-hand side of \eqref{eq:gn} are $>2$. Hence, $|\nabla u_i|_2^2\geq a_2>0$ for every $u_i\in\mathscr K_i$. 

In the case $\lambda_{ij}<0$, \eqref{eq:H_bounds} simplifies to
\begin{equation} \label{ineq2}
\irn|\nabla u_i|^2+\kappa_i r_i^2 \le \mu_i|u_i|_{p+1}^{p+1},
\end{equation}
so the second term on the right-hand side of \eqref{eq:gn} can be omitted and therefore the assumption $\alpha_{ij}>1+\frac4N$ is unnecessary.

To prove that $\mathcal Q_i$ is bounded away from $0$ we argue by contradiction. Assume that $(\bar\kappa_n,\bar u_n)$ solves \eqref{eq:system} for some $t_n\in[0,1]$ with $\bar u_n\in\mathscr U^{t_n}$ and that $\kappa_{n,i}\to 0$. As $(\kappa_{n,i})$ is bounded in $\r$ and $(u_{n,i})$ is bounded in $\Hr$, passing to a subsequence, for each $i$ we have that $\kappa_{n,i}\to\kappa_{0,i}$ in $\r$, $u_{n,i}\rh u_{0,i}$ weakly in $\Hr$, $ u_{n,i}\to u_{0,i}$ strongly in $L^q(\rn)$ for $q\in(2,\frac{2N}{N-2})$ (see \cite[Compactness Lemma]{st} or \cite[Corollary 1.26]{w}) and $t_n\to t_0\in[0,1]$. As a consequence, $u_{0,i}\neq 0$ because otherwise
$$
0<a_2\leq\irn|\nabla u_{n,i}|^2\leq \mu_i\irn|u_{n,i}|^{p+1}+ C_4\irn|u_{n,i}|^{\alpha_{ij}+1}=o(1).
$$
It is straightforward to verify that $(\bar\kappa_0,\bar u_0)$ solves \eqref{eq:system} for $t_0$. But $\kappa_{0,i}=0$, contradicting Lemma \ref{lem:positive} or Lemma \ref{lem:positive2}. This completes the proof.
\end{proof}

\begin{lemma} \label{lem:compact}
$\mathscr K_i$ is compact for all $i=1,\ldots,\ell$.
\end{lemma}

\begin{proof}
Let $(\bar\kappa_n,\bar u_n)$ be a solution to \eqref{eq:system} with $\bar u_n\in\mathscr{U}^{t_n}$ for $t_n\in[0,1]$. Passing to a subsequence we have that $\kappa_{n,i}\to\kappa_{0,i}>0$ in $\r$, $u_{n,i}\rh u_{0,i}$ weakly in $\Hr$ and $t_n\to t_0\in[0,1]$. Then $|u_{0,i}|_2\leq r_i$, and arguing as in the proof of Lemma \ref{lem:H_bounds} we see that $(\bar\kappa_0,\bar u_0)$ solves \eqref{eq:system} for $t_0$. Therefore,
\begin{align*}
0&=\irn\nabla u_{n,i}\cdot\nabla(u_{n,i}-u_{0,i})+\kappa_{n,i}\irn u_{n,i}(u_{n,i}-u_{0,i})-\irn f_i^{t_n}(\bar u_n)(u_{n,i}-u_{0,i}) \\
&\quad-\irn\nabla u_{0,i}\cdot\nabla(u_{n,i}-u_{0,i})-\kappa_{0,i}\irn u_{0,i}(u_{n,i}-u_{0,i})+\irn f_i^{t_0}(\bar u_0)(u_{n,i}-u_{0,i}).
\end{align*}
Hence, 
$$
0=\lim_{n\to\infty}\left(\irn|\nabla(u_{n,i}-u_{0,i})|^2+\kappa_{0,i}\irn |u_{n,i}-u_{0,i}|^2\right).
$$
As $\kappa_{0,i}>0$, we conclude that $|u_{0,i}|_2= r_i$ and  $u_{n,i}\to u_{0,i}$ strongly in $\Hr$. Thus $u_{0,i}\in\mathscr K_i$.
\end{proof}

Recall the notation introduced at the end of Section \ref{intro}. Note that $\mathscr S_i =G^{-1}(r_i^2)$ where $G:\Hr\to\r$ is given by $G(u):=\irn u^2$. So $\mathscr S_i$ is a smooth Hilbert submanifold of $\Hr$ and its tangent space at $u\in\mathscr S_i$ is
$$T_u(\mathscr S_i):=\Big\{v\in\Hr:\irn uv=0\Big\}.$$
Let $\wt S_i^t: \cH\to\Hr$ be defined by
\begin{equation} \label{sit}
\langle \wt S_i^t(\bar u),v\rangle = \irn \nabla u_i\cdot\nabla v + \kappa_i^t\irn u_iv -\irn f_i^t(\bar u)v \qquad\text{for all \ }v\in\Hr,
\end{equation}
where $\kappa_i^t=\kappa_i^t(\bar u)$ is given by \eqref{kappa}. Then 
\[
\langle\wt S_i^t(\bar u),u_i\rangle=0\qquad\text{for every \ }\bar u=(u_1,\ldots,u_\ell)\in\mathscr{T}, \ i=1,\ldots,\ell.
\]
Setting $\bar\kappa^t:= (\kappa_1^t,\ldots,\kappa_\ell^t)$ and $\wt S^t:= (\wt S_1^t,\ldots,\wt S_\ell^t)$, it is clear that $(\bar\kappa^t,\bar u)$ is a solution to \eqref{eq:constraint}-\eqref{eq:system} if and only if $\bar u\in\mathscr{T}$ and $\wt S^t(\bar u)=0$. 

Let $Q_{u_i}:\Hr\to T_{u_i}(\mathscr{S}_i)$ be the orthogonal projection and define
\begin{equation} \label{S}
S^t:= (S_1^t,\ldots,S_\ell^t):\mathscr{T}\to T_{\bar u}(\mathscr{T}),\qquad\text{with}\qquad S_i^t(\bar u):=Q_{u_i}\wt S_i^t(\bar u).
\end{equation}

\begin{proposition}\label{prop:S}
$(\bar\kappa^t,\bar u)$ is a solution to \eqref{eq:constraint}-\eqref{eq:system} if and only if  $S^t(\bar u)=0$. 
\end{proposition}

\begin{proof}
It suffices to show that $S^t(\bar u)=0$ if and only if $\wt S^t(\bar u)=0$. Clearly, $S^t(\bar u)=0$ if $\wt S^t(\bar u)=0$. On the other hand, if $S_i^t(\bar u)=0$ then $\langle\wt S_i^t(\bar u),v\rangle=0$ for every $v\in T_{u_i}(\mathscr{S}_i)$. Since $\Hr=T_{u_i}(\mathscr{S}_i)\oplus\r u_i$ and $\langle\wt S_i^t(\bar u),u_i\rangle=0$, we have that $\langle\wt S_i^t(\bar u),v\rangle=0$ for every $v\in \Hr$.
\end{proof}

Let $(E_{k,i})_{k\ge 1}$ be an ascending sequence of linear subspaces of $\Hr$ (to be fixed later) such that $\dim E_{k,i} = k$ and $\overline{\bigcup_{k\ge1} E_{k,i}} = \Hr$, and put 
\[
E_k := E_{k,1} \times \cdots\times E_{k,\ell}\subset\cH, \qquad P_{k,i}: \Hr\to E_{k,i}, \qquad P_{k}: \cH\to E_{k}
\]
where $P_{k,i}, P_k$ are the orthogonal projections. We shall make repeated use of the following result.
 
\begin{lemma} \label{lem:convergence}
Let $(\bar u_n)$ be a bounded sequence in $\mathscr{T}$ such that $\bar u_n\in \cU^{t_n}\cap E_{k_n}$ where $k_n\to\infty$ and $P_{k_n}S^{t_n}(\bar u_n)=0$. \ If $t_n\to t$, \ $\cW\subset \ol{\mathscr{U}^t}$ \ and the distance $\mathrm{dist}(\bar u_n,\cW)\to 0$, then there exists a solution $(\bar\kappa^t,\bar u)$ to \eqref{eq:constraint}-\eqref{eq:system} such that $\bar u\in \ol\cW$. 
\end{lemma}

\begin{proof}
Set $\kappa_{n,i}^{t_n}:=\kappa_{i}^{t_n}(\bar u_n)$. Since $\langle\wt S_i^{t_n}(\bar u_n),u_{n,i}\rangle = 0$, we have
\begin{equation} \label{eq:kappa_bdd}
0<\kappa_{n,i}^{t_n}r_i^2\leq\irn|\nabla u_{n,i}|^2 +\kappa_{n,i}^{t_n}\irn u_{n,i}^2=\irn f_i^{t_n}(\bar u_n)u_{n,i}.
\end{equation}
As $(\bar u_n)$ is bounded in $\cH$, it follows that $\kappa_{n,i}^{t_n}$ is bounded and, passing to a subsequence, $\kappa_{n,i}^{t_n}\to \wh\kappa_i\ge0$. Furthermore, we derive from \eqref{sit} that $(\wt S_i^{t_n}(\bar u_n))$ is bounded in $\Hr$.

Let $v\in\Hr$. Then $P_{k_n,i}v=v_{n}+\alpha_{n}u_{n,i}$ with $v_{n}\in T_{u_{n,i}}(\mathscr{S}_i)\cap E_{k_n,i}$ and $\alpha_n\in\r$. Since, by assumption, $P_{k_n}S^{t_n}(\bar u_n)=0$, we have that
\begin{equation*}
\langle\wt S_i^{t_n}(\bar u_n),P_{k_n,i}v\rangle=\langle\wt S_i^{t_n}(\bar u_n),v_n\rangle=\langle S_i^{t_n}(\bar u_n),v_n\rangle=\langle P_{k_n,i}S_i^{t_n}(\bar u_n),v_n\rangle=0.
\end{equation*}
As $P_{k_n,i}v\to v$ and $(\wt S_i^{t_n}(\bar u_n))$ is bounded in $\Hr$, it follows that
\begin{equation} \label{eq:limit_Si}
0=\lim_{n\to\infty}\langle \wt S_i^{t_n}(\bar u_n),v\rangle = \lim_{n\to\infty}\left(\irn\nabla u_{n,i}\cdot \nabla v +\kappa_{n,i}^{t_n}\irn u_{n,i}v-\irn f_i^{t_n}(\bar u_n)v\right)
\end{equation}
for every $v\in\Hr$. \ Passing to a subsequence again, $u_{n,i}\rh u_i$ weakly in $\Hr$ and $u_{n,i}\to u_i$ strongly in $L^q(\rn)$ for all $2<q<\frac{2N}{N-2}$ and all $i$ (\cite[Compactness Lemma 2]{st}, \cite[Corollary 1.26]{w}). Thus, from \eqref{eq:limit_Si} we get that either $u_i=0$, or $\wh\kappa_i=\kappa_i^t(\bar u)$ and $\wt S_i^t(\bar u)=0$, where $\bar u=(u_1,\ldots,u_\ell)$. 

Since $(u_{n,i})$ are bounded in $\Hr$ and $u_{n,i}\ne 0$, from the equality in \eqref{eq:kappa_bdd} we derive the inequalities  \eqref{eq:H_bounds} and \eqref{eq:gn} for $u_{n,i}$ and, thus, conclude that $(|\nabla u_{n,i}|_2)$ is bounded away from $0$ (if $\lambda_{ij}<0$, the summation term in \eqref{eq:gn} should be omitted as follows from \eqref{ineq2}). Therefore, as
\begin{align*}
0<a_3 & \le \lim_{n\to\infty}\left(\irn|\nabla u_{n,i}|^2 +\kappa_{n,i}^{t_n}\irn u_{n,i}^2\right) \\
& = \lim_{n\to\infty}\irn f_i^{t_n}(\bar u_n)u_{n,i} = \irn f_i^t(\bar u)u_i = \irn|\nabla u_i|^2 +\wh\kappa_{i}\irn u_i^2,
\end{align*}
we have that $u_i\ne 0$. But then, by Lemma \ref{lem:positive} or Lemma \ref{lem:positive2}, $\wh\kappa_i=\kappa_i^t(\bar u)>0$, and as a consequence, $u_{n,i}\to u_i$ in $\Hr$. So $\bar u_n\to\bar u$ and, as $\mathrm{dist}(\bar u_n,\cW)\to 0$, it follows that $\bar u\in\ol\cW$.
\end{proof}

\section{Computations for the uncoupled system} \label{comput}

If $t=0$, the system \eqref{eq:system} is uncoupled and, for each $i\in\{1,\ldots,\ell\}$, the $i$-th components of its solutions satisfying \eqref{eq:constraint} are the critical points of the functional
\[
\cJ_i: \mathscr S_i\to\r,\qquad  \cJ_i(u_i) := \frac12\irn|\nabla u|^2 - \frac1{p+1}\irn\mu_i(u^+)^{p+1}.
\]
Let $I_i:\Hr\to\r$ be given by
$$
I_i(u):=\frac12\irn|\nabla u|^2 + \frac{\kappa_i}{2}\irn u^2 - \frac1{p+1}\irn\mu_i(u^+)^{p+1}.
$$
The critical points of $I_i$ are the solutions to the problem
\begin{equation} \label{eq:equation}
-\Delta u + \kappa_i u = \mu_i (u^+)^p, \qquad u\in H^1_\mathrm{rad}(\rn).
\end{equation}
If $u\in\mathscr S_i$, then $u$ is a critical point of $I_i$ if and only if $u$ is a critical point of $\cJ_i$ and
\begin{equation} \label{kappai}
\kappa_i = r_i^{-2}\left(\irn\mu_i(u^+)^{p+1} - \irn|\nabla u|^2\right),
\end{equation}
see Proposition \ref{prop:hessian}$\,(i)$.

The identity in the next lemma goes back to \cite[Lemma 2.7]{jj} where it has been shown for more general right-hand sides. The argument here is the same as in \cite{jj} but we provide it for the reader's convenience. 

\begin{lemma} \label{lem:NP}
All nontrivial solutions to \eqref{eq:equation} satisfy the identity
\[
\irn|\nabla u|^2 = \frac N2\,\frac{p-1}{p+1}\irn \mu_i (u^+)^{p+1}.
\]
Moreover, $\kappa_i>0$ for such solutions.
\end{lemma}

\begin{proof}
Each solution $u$ satisfies
\[
\irn|\nabla u|^2 + \kappa_i\irn u^2 = \irn \mu_i (u^+)^{p+1}
\]
as well as the Pohozaev identity (see e.g. \cite[Appendix B]{w}) 
\[
\tfrac{N-2}2\irn|\nabla u|^2 + \tfrac N2\kappa_i\irn u^2 = \tfrac N{p+1}\irn\mu_i (u^+)^{p+1}.  
\]
Eliminating $\kappa_i$ gives the first conclusion. Eliminating the right-hand sides gives
\[
(1-\tfrac{p+1}{2^*})\irn|\nabla u|^2 = (\tfrac{p+1}2-1)\,\kappa_i\irn u^2,
\]
so $\kappa_i>0$. Here $\frac{p+1}{2^*} := 0$ if $N=2$.
\end{proof}

It is well known that there exists a unique positive solution $\omega_0$ to the problem \eqref{eq:equation}  with $\kappa_i=1$ \cite{kw, cl} and it is a nondegenerate critical point of $I_i:\Hr\to\r$ \cite[Appendix C]{nt}. Setting
\[
\omega_i(x):=\gamma_i^\frac{2}{p-1}\omega_0(\gamma_ix) \quad \text{and} \quad \kappa_i := \gamma_i^{2},
\]
we see that $\omega_i$ solves \eqref{eq:equation} and is unique and nondegenerate as well. As $|\omega_i|_2^2 = \gamma_i^{\frac 4{p-1}-N}|\omega_0|_2^2$ and $p\ne 1+\frac4N$, \ $|\omega_i|_2$ can be prescribed arbitrarily. From now on we assume $\gamma_i$ has been chosen so that $\omega_i\in\mathscr S_i$. 

\begin{proposition} \label{nondeg}
$\omega_i$ is the unique critical point of $\cJ_i: \mathscr S_i\to\r$. Moreover, it is nondegenerate and its Morse index is $1$.
\end{proposition}
\begin{proof}
Since $\omega_i$ is the only positive solution for \eqref{eq:equation}, it follows from Lemma \ref{lem:positive} that it is the only critical point of $I_i$, and by Proposition \ref{prop:hessian}(i), also of $\cJ_i$ (here the conditions on the coupling terms are irrelevant).

It is easy to see and is well known that $\omega_i$ is a mountain pass point for $I_i$, hence the Morse index of $\omega_i$ with respect to $I_i$ is $1$, see e.g. \cite[Corollary II.3.1]{ch}. 

Let 
\begin{equation} \label{sstar}
(s\star u)(x) := e^{Ns/2}u(e^sx),\quad s\in\r.
\end{equation}
Then $s\star u\in \mathscr{S}_i$ for all $s$ and
\begin{align*}
I_i(s\star u) & = \frac12e^{2s}\irn|\nabla u|^2 + \kappa_i\irn u^2 - \frac1{p+1}e^{Ns(p-1)/2}\irn \mu_i(u^+)^{p+1}, \\
\cJ_i(s\star u) & = \frac12e^{2s}\irn|\nabla u|^2 - \frac1{p+1}e^{Ns(p-1)/2}\irn \mu_i(u^+)^{p+1}. 
\end{align*}
The family of functions \eqref{sstar} has been introduced in \cite{jj}.
Differentiating and using Lemma \ref{lem:NP}, we obtain
\[
\frac d{ds}\Big|_{s=0}I_i(s\star\omega_i) = \frac d{ds}\Big|_{s=0}\cJ_i(s\star\omega_i) = \irn|\nabla \omega_i|^2 - \frac N2\,\frac{p-1}{p+1}\irn \mu_i\omega_i^{p+1} = 0.
\]
As $p>1+4/N$, it follows from the equality above that
\begin{align*}
\frac {d^2}{ds^2}\Big|_{s=0}I_i(s\star\omega_i) & = \frac {d^2}{ds^2}\Big|_{s=0}\cJ_i(s\star\omega_i) = 2\irn|\nabla \omega_i|^2 - \frac {N^2}4\,\frac{(p-1)^2}{p+1}\irn \mu_i\omega_i^{p+1} \\
& = \left(N\,\frac{p-1}{p+1} - \frac {N^2}4\,\frac{(p-1)^2}{p+1}\right) \irn\mu_i\omega_i^{p+1} < 0.
\end{align*} 
Setting $c(s) := s\star \omega_i$ and using Proposition \ref{prop:hessian}$\,(ii)$, we also have
\[
\frac {d^2}{ds^2}\Big|_{s=0}I_i(c(s)) = I_i''(c(0))[\dot c(0),\dot c(0)] + I_i'(c(0))\ddot c(0) = I_i''(c(0))[\dot c(0),\dot c(0)] = \nabla^2 \cJ_i[\dot c(0),\dot c(0)]|_{\omega_i}.
\]
So $\nabla^2 \cJ_i[\dot c(0),\dot c(0)]|_{\omega_i} < 0$ and the Morse index of $\omega_i$ with respect to $\cJ_i$ is positive. 

Let $L:=\mathbf{H}I_i(\omega_i)$ be the Hessian of $I_i$ at $\omega_i$, i.e. $L$ is the linear mapping representing $I''_i(\omega_i)$, see Example \ref{example_hessian} (here and in what follows we use the term Hessian both for the quadratic form and for the linear mapping representing it). By Proposition \ref{prop:hessian}$\,(ii)$, $Q_{\omega_i}\circ L|_{T_{\omega_i}(\mathscr{S}_i)}$ is the covariant Hessian of $\cJ_i$ at $\omega_i$ where $Q_{\omega_i}: H^1_{\text{rad}}(\rn)\to T_{\omega_i}(\mathscr{S}_i)$ is the orthogonal projection. It remains to show that the Morse index of $\omega_i$ with respect to $\cJ_i$ is $1$ and $Q_{\omega_i}\circ L|_{T_{\omega_i}(\mathscr{S}_i)}$ has trivial kernel. For $j=1,2$, let
\begin{align*}
\lambda_j  : = \sup_{\substack{A\subset H^1_{\text{rad}}(\rn)\\ \text{codim\,}A\le j-1}}\inf_{\substack{u\in A\\ \|u\|=1}}\langle Lu,u\rangle \qquad\text{and}\qquad
\wt\lambda_j  : = \sup_{\substack{A\subset T_{\omega_i}(\mathscr{S}_i)\\ \text{codim\,} A\le j-1}}\inf_{\substack{u\in A\\ \|u\|=1}}\langle Lu,u\rangle.
\end{align*}
Notice that, if $A$ is a subspace of codimension $\le j-1$ of $\Hr$, then $A\cap T_{\omega_i}(\mathscr{S}_i)$ is a subspace of codimension $\le j-1$ of $T_{\omega_i}(\mathscr{S}_i)$ and
$$\inf_{\substack{u\in A\\ \|u\|=1}}\langle Lu,u\rangle\leq\inf_{\substack{u\in A\cap T_{\omega_i}(\mathscr{S}_i)\\ \|u\|=1}}\langle Lu,u\rangle\leq\wt\lambda_j.$$
Hence $\lambda_j\le\wt\lambda_j$. As $L$ is of the form \emph{identity minus compact}, it has no essential spectrum. Thus $\lambda_j$ and $\wt\lambda_j$ are eigenvalues as follows from
\cite[Theorem XIII.1]{rs}.
Since the Morse index of $\omega_i$ with respect to $\cJ_i$ is positive, $\wt\lambda_1 < 0$. As $\omega_i$ is a nondegenerate critical point of index 1 for $I_i$, $0 < \lambda_2\le \wt\lambda_2$  and therefore it is also nondegenerate of index 1 for $\cJ_i$.
\end{proof}

\begin{remark} \label{rem:morse}
\emph{
It follows from the proposition above that $\omega_i$ is a mountain pass point for $\cJ_i$. This is in fact known, see \cite{jj}. However, our proposition gives more precise information which will be needed below.
}
\end{remark}

The stereographic projection $\sigma_i:\mathscr S_i\smallsetminus\{-\omega_i\}\to T_{\omega_i}(\mathscr S_i)$ from the point $-\omega_i$ given by
\begin{equation}\label{eq:stereographic}
\sigma_i(u)=\frac{r_i^2u-\big(\irn \omega_iu\big)\,\omega_i}{r_i^2+\irn \omega_iu},\qquad u\in\mathscr S_i\smallsetminus\{-\omega_i\},
\end{equation}
is a diffeomorphism whose inverse is
$$\sigma_i^{-1}(v)=\frac{2r_i^2v+(r_i^2-|v|_2^2)\,\omega_i}{r_i^2+|v|_2^2},\qquad v\in T_{\omega_i}(\mathscr S_i),$$
where $|v|_2$ is the $L^2$-norm of $v$. So $\sigma_i^{-1}$ is a chart for the submanifold $\mathscr S_i$ of $\Hr$ at $\omega_i$, and $\sigma_i(\omega_i)=0$. By Proposition \ref{nondeg}, $0$ is a nondegenerate critical point of $\Phi_i:=\cJ_i\circ \sigma_i^{-1}$ of Morse index $1$. So the Hessian $\mathbf{H}\Phi_i(0):T_{\omega_i}(\mathscr S_i)\to T_{\omega_i}(\mathscr S_i)$ of $\Phi_i$ at $0$, defined by
$$\langle\mathbf{H}\Phi_i(0)v,w\rangle=\Phi_i''(0)[v,w]\qquad\text{for all \ }v,w\in T_{\omega_i}(\mathscr S_i),$$
is invertible, $T_{\omega_i}(\mathscr S_i)$ is the orthogonal sum in $\Hr$ of the subspaces $F_i^-$ and $F_i^+$ where the Hessian is negative and positive definite respectively, and $\dim F_i^-= 1$.

We fix an ascending sequence $(F_{k,i})_{k\ge 2}$ of linear subspaces of $T_{\omega_i}(\mathscr S_i)$ such that $F_{2,i}:= F_i^-$, $F_{k,i}:=F_{2,i}\oplus \wt F_{k,i}$ ($k\ge 3$) where $\wt F_{k,i}\subset F_i^+$, $\dim F_{k,i} = k-1$ and  $\overline{\bigcup_{k\ge2} F_{k,i}} = T_{\omega_i}(\mathscr S_i)$. Let $E_{1,i}$ be the orthogonal complement of $T_{\omega_i}(\mathscr S_i)$ in $\Hr$ and set $E_{k,i} := F_{k,i}\oplus E_{1,i}$ for $k\ge 2$. Then, $\dim E_{k,i} = k$ and $\overline{\bigcup_{k\ge1} E_{k,i}} = \Hr$. Furthermore, we assume that $\wt F_{3,i}$ has been chosen in such a way that $\omega_i\in E_{3,i}$. Note that $\sigma_i$ maps $(\mathscr S_i\smallsetminus\{-\omega_i\})\cap E_{k,i}$ into $F_{k,i}$ for each $k\ge 3$. 

Let $P_{k,i}:\Hr\to E_{k,i}$ be the orthogonal projection. Put
\[
\cW_i := \{v_i\in \mathscr{S}_i: \kappa_i>0\}, \qquad \text{where \ }\kappa_i := r_i^{-2}\left(\irn\mu_i(v_i^+)^{p+1} - \irn|\nabla v_i|^2\right),
\]
and notice that $\cW_1\times\cdots\times \cW_\ell = \mathscr{U}^0$ (see \eqref{eq:ui} for the definition of $\mathscr{U}^0$),

\begin{lemma} \label{lem:nonexistence}
For each open bounded neighborhood $\cV\subset\sigma_i(\cW_i)$ of $0$ in $T_{\omega_i}(\mathscr S_i)$ there exists $k_\cV\in\n$ such that, for all $k\ge k_\cV$, $0$ is the only critical point of $\Phi_i|_{F_{k,i}}$ in $\cV\cap F_{k,i}$.
\end{lemma}

\begin{proof}
Set $A_i:=\mathbf{H}\Phi_i(0)$ and $A_{k,i}:=P_{k,i}\circ A_i|_{F_{k,i}}:F_{k,i}\to F_{k,i}$. Note that $A_{i}(F_{k,i})$ is orthogonal to $E_{1,i}$ because it is contained in $T_{\omega_i}(\mathscr S_i)$ and, therefore, the range of $A_{k,i}$ is indeed in $F_{k,i}$.

As $A_i(F_i^-)\subset F_i^-$ and $A_i(F_i^+)\subset F_i^+$, $A_{k,i}$ is an isomorphism. Indeed, if $u\in F_{k,i}$, $u=v+w$ with $v\in F_i^-=F_{2,i}$, $w\in \wt F_{k,i}$ and $A_{k,i}u=0$, then $0=\langle A_{k,i}u,v\rangle=\langle A_iv,v\rangle$ and $0=\langle A_{k,i}u,w\rangle=\langle A_iw,w\rangle$ and, as a consequence, $v=0$ and $w=0$. This shows that $A_{k,i}$ is an isomorphism. Moreover, as there exists $a>0$ such that $\langle A_iv,v\rangle\le -a\|v\|^2$ for all $v\in F_i^-$ and $\langle A_iv,v\rangle\ge a\|v\|^2$ for all $v\in F_i^+$, we have that $\|A_{k,i}^{-1}\|\leq a^{-1}$ for all $k\geq 3$.

Next we show that there exist $k_0\in\n$ and a neighborhood $\cO_i$ of $0$ in $T_{\omega_i}(\mathscr S_i)$ such that, for all $k\ge k_0$, $0$ is the only critical point of $\Phi_i|_{F_{k,i}}$ in $\cO_i\cap F_{k,i}$. 

Arguing by contradiction, assume there are subsequences $(k_n)$ and $(v_n)$ such that $k_n\to\infty$, $v_n\in F_{k_n,i}\smallsetminus\{0\}$, \ $v_n\to 0$ \ and \ $P_{k_n,i}\nabla\Phi_i(v_n)=0$. Since $A_{k_n,i}v_n = A_{k_n,i}v_n-P_{k_n,i}\nabla\Phi_i(v_n)$, we have
\begin{equation} \label{fixed}
v_n = v_n - A_{k_n,i}^{-1}P_{k_n,i}\nabla\Phi_i(v_n).
\end{equation}
Let $K_n:F_{k_n,i}\to F_{k_n,i}$ be given by $K_n(w):=w - A_{k_n,i}^{-1}P_{k_n,i}\nabla\Phi_i(w)=A_{k_n,i}^{-1}(A_{k_n,i}w-P_{k_n,i}\nabla\Phi_i(w))$. As $\Phi_i''$ is continuous, there exists $\delta>0$ such that 
$$\|A_{k_n,i}-P_{k_n,i}\circ\mathbf{H}\Phi_i(w)\|=\|P_{k_n,i}\circ\mathbf{H}\Phi_i(0)-P_{k_n,i}\circ\mathbf{H}\Phi_i(w)\| \le \tfrac{a}{2}\qquad\text{whenever \ }\|w\|\le\delta.$$
Hence,
\[
\|K'_n(w)\| = \|A_{k_n,i}^{-1}\circ(A_{k_n,i}-P_{k_n,i}\circ\mathbf{H}\Phi_i(w))\| \le \tfrac12 \qquad \text{for all \ } \|w\|\le\delta,
\] 
and therefore 
\[
\|K_n(v_n)\| = \left\|\int_0^1\frac{d}{dt}K_n(tv_n)\,dt\right\| \le \int_0^1\|K_n'(tv_n)v_n\|\,dt \le \tfrac12\|v_n\|
\]
for all $n$ large enough. So $\|v_n\|\le \frac12\|v_n\|$ according to \eqref{fixed}. Thus $v_n=0$, a contradiction. Hence $k_\cV$ and $\cO_i$ exist as claimed.

We complete the proof by showing that, after possibly taking a larger $k_\cV$, $\Phi_i|_{F_{k,i}}$ does not have critical points in $(\cV\smallsetminus\cO_i)\cap F_{k,i}$. Assuming the contrary, we find subsequences $(k_n)$ and $(v_n)$ with $k_n\to\infty$, $v_n\in (\cV\smallsetminus\cO_i)\cap F_{k,i}$ and $P_{k_n,i}\nabla\Phi_i(v_n)=0$. Then $u_n := \sigma^{-1}_i(v_n)\in\mathscr S_i\cap E_{k_n,i}$ and $u_n$ is a critical point of $\cJ_i|_{E_{k_n,i}}$. Hence, it is also a critical point of $I_i|_{E_{k_n,i}}$, i.e., $P_{k_n,i}S_i^0(u_n)=0$, where $S_i^0$ is defined in \eqref{sit}-\eqref{S}. It follows from Lemma \ref{lem:convergence} that there exists a critical point $u\in \sigma^{-1}(\ol{\cV}_i\smallsetminus\cO_i)$ of $\cJ_i$. According to Proposition \ref{nondeg}, $u=\omega_i$ which is impossible because $u\not\in \sigma^{-1}(\cO_i)$.   
\end{proof}

Let $\bar\omega := (\omega_1,\ldots,\omega_\ell)$ and $\Phi:T_{\bar\omega}(\mathscr T)\to\r$ be the functional given by $\Phi(\bar v) := \Phi_1(v_1)+\cdots+\Phi_\ell(v_\ell)$ where  $\bar v:=(v_1,\ldots,v_\ell)\in T_{\bar\omega}(\mathscr T)$. Then, $\nabla(\Phi|_{F_k})=(P_k\circ\nabla\Phi)|_{F_k}$ and
\[
(P_k\circ\nabla\Phi)|_{F_k}(\bar v) = \left((P_{k,1}\circ\nabla\Phi_1)|_{F_{k,1}}(v_1),\ldots, (P_{k,\ell}\circ\nabla\Phi_\ell)|_{F_{k,\ell}}(v_\ell)\right),
\]
where $F_k:=F_{k,1}\times\cdots\times F_{k,\ell}\subset T_{\bar\omega}(\mathscr T)$ and $P_k:\cH\to E_k:=E_{k,1}\times\cdots\times E_{k,\ell}$ is the orthogonal projection.

\begin{corollary} \label{cor:degree}
For each open bounded neighborhood $\cV\subset\sigma(\mathscr{U}^0)$ of $0$ in $T_{\bar\omega}(\mathscr T)$ there exists $k_\cV\in\n$ such that the Brouwer degree $\deg((P_k\circ\nabla\Phi)|_{F_k}, \cV_0\cap F_{k}, 0) = (-1)^\ell$  for all $k\ge k_\cV$ and every open neighborhood $\cV_0$ of $0$ such that $\cV_0\subset\cV$. 
\end{corollary}

Basic properties of Brouwer's degree may be found e.g. in \cite{am, dm, w}.

\begin{proof}
Choose bounded $\cV_i\subset\sigma_i(\cW_i)$ in such a way that $\cV\subset \cV_1\times\cdots\times \cV_\ell$. According to Lemma \ref{lem:nonexistence}, we may find $k_\cV$ such that $0$ is the only critical point of $\Phi_i$ in $\cV_i\cap F_{k,i}$ for each $i$ and every $k\ge k_\cV$, and hence $0$ is the only critical point for  $\Phi$ in $\cV\cap F_k$. In the proof of Lemma \ref{lem:nonexistence} we have shown that, for each $i$, \ $A_{k,i}:=(P_{k,i}\circ \mathbf{H}\Phi_i(0))|_{F_{k,i}}:F_{k,i}\to F_{k,i}$ is an isomorphism having exactly one negative eigenvalue. So the linear mapping $(A_{k,1},\ldots, A_{k,\ell})$ has $\ell$ negative eigenvalues.  Since, by the excision property, the degree is independent of the choice of $\cV_0$, this implies that $\deg((P_k\circ\nabla\Phi)|_{F_k}, \cV_0\cap F_{k}, 0) = (-1)^\ell$, as claimed.
\end{proof}

\section{Proof of Theorems \ref{mainthm} and \ref{mainthm2}} \label{pfmain}

Let $\sigma_i: \mathscr{S}_i\smallsetminus \{-\omega_i\} \to T_{\omega_i}(\mathscr S_i)$ be the stereographic projection defined in \eqref{eq:stereographic} and set $\sigma(\bar v) := (\sigma_1(v_1),\ldots,\sigma_\ell(v_\ell))\in T_{\omega_1}(\mathscr S_1)\times\cdots\times T_{\omega_\ell}(\mathscr S_{\ell})= T_{\bar\omega}(\mathscr T)$ if $v_i\in\mathscr{S}_i\smallsetminus \{-\omega_i\}$ for all $i$.  

\begin{lemma}
$\sigma(\mathscr{U}^t\cap\cW)$ is bounded in $\cH$ uniformly with respect to $t$ for any bounded subset $\cW$ of~$\cH$.
\end{lemma}

\begin{proof}
Let $\cW$ be bounded in $\cH$. It suffices to show that $\cW_i^t:=\{u_i:\bar u\in \mathscr{U}^t\cap\cW\}$ is $L^2$-bounded away from $-\omega_i$ uniformly in $t$ for each $i=1,\ldots,\ell$. Indeed, if this is the case, then the angle between $u_i$ and $-\omega_i$ is uniformly bounded away from 0, so
\[
r_i^2+\irn\omega_iu_i = |\!-\omega_i|_2^2\,|u_i|_2^2- \irn(-\omega_i)u_i \ge a > 0
\]
for every $u_i\in\cW_i^t$ and, as a consequence, 
\begin{equation*}
\|\sigma_i(u_i)\|=\frac{\|r_i^2u_i-\big(\irn \omega_iu_i\big)\,\omega_i\|}{r_i^2+\irn \omega_iu_i}\leq C\|u_i\|\qquad\forall u_i\in\cW_i^t.
\end{equation*}
Arguing by contradiction, we find $\bar u_n\in\mathscr{U}^{t_n}\cap\cW$ such that $u_{n,i}\to -\omega_i$ in $L^2(\rn)$. Passing to a subsequence, $t_n\to t$, $u_{n,i}\rh u_i$ weakly in $\Hr$, $u_{n,i}\to u_i$ strongly in $L^q(\rn)$ for all $2< q < \frac{2N}{N-2}$ and $u_{n,i}\to u_i$ a.e. in $\rn$. Hence $u_i=-\omega_i$. Since $f_i^t(\bar u)=0$ if $u_i\le 0$ (in particular, if $u_i=-\omega_i$), it follows that
\begin{align*}
0&\le \limsup_{n\to\infty} \kappa_i^{t_n}(\bar u_n) = \limsup_{n\to\infty} r_i^{-2}\left(\irn f_i^{t_n}(\bar u_n)u_{n,i} - \irn|\nabla u_{n,i}|^2\right)\\
&\leq\limsup_{n\to\infty} \left(-r_i^{-2} \irn|\nabla u_{n,i}|^2\right) \le -r_i^{-2}\irn|\nabla\omega_i|^2 < 0.
\end{align*}
This shows that $\cW_i^t$ is $L^2$-bounded away from $-\omega_i$ uniformly in $t$ for every $i=1,\ldots,\ell$.
\end{proof}
\medskip

\begin{proof}[Proof of Theorems \ref{mainthm} and \ref{mainthm2}]
According to Lemma \ref{lem:positive} or Lemma \ref{lem:positive2}, it suffices to show that \eqref{eq:system} with $t=1$ has a solution in $\mathscr{U}^1$. 

Consider the map
\begin{equation} \label{gtk}
G^t_k(\bar v) := (\sigma'(\sigma^{-1}(\bar v))\circ P_k\circ S^t\circ\sigma^{-1})(\bar v), \qquad \bar v\in\sigma(\mathscr{U}^t\cap E_k)\subset T_{\bar\omega}(\mathscr T)\cap F_k,
\end{equation}
where $S^t$ is defined in \eqref{S} and, as in the previous section, $F_k:=F_{k,1}\times\cdots\times F_{k,\ell}\subset T_{\bar\omega}(\mathscr T)$, $E_k:=E_{k,1}\times\cdots\times E_{k,\ell}\subset \cH$ and $P_k:\cH\to E_k$ is the orthogonal projection. Note that
\[
G^t_k(\bar v)=0 \quad \text{if and only if}\quad P_kS^t(\bar u)=0,
\]
where $\bar u = \sigma^{-1}(\bar v)$. According to Proposition \ref{prop:S} we must show that there exists $\bar u\in\cU^1$ such that $S^1(\bar u)=0$.

For $v_i\in F_{k,i}$, let $L_{k,i}(v_i):F_{k,i}\to F_{k,i}$ be the unique linear operator satisfying 
$$\langle L_{k,i}(v_i)z,w\rangle=\langle (\sigma_i^{-1})'(v_i)z,\,(\sigma_i^{-1})'(v_i)w\rangle\qquad\forall z,w\in F_{k,i}.$$
Recall that $\Phi_i = \cJ_i\circ\sigma_i^{-1}$. Then, for $\bar v$ as above and any $w\in F_{k,i}$ we have
\begin{align*}
\langle\nabla\Phi_i(v_i),w\rangle &=\Phi_i'(v_i)w=\cJ'_i(\sigma_i^{-1}(w))[(\sigma_i^{-1})'(v_i)w]=\langle\nabla\cJ_i(\sigma_i^{-1}(w)),(\sigma_i^{-1})'(v_i)w\rangle \\
&=\langle P_{k,i}S_i^0(\sigma_i^{-1}(v_i)),\,(\sigma_i^{-1})'(v_i)w\rangle=\langle (\sigma_i^{-1})'(v_i)[G_{k,i}^0(v_i)],\,(\sigma_i^{-1})'(v_i)w\rangle \\
&=\langle L_{k,i}(v_i)G^0_{k,i}(v_i),w\rangle.
\end{align*}
Therefore, $(P_{k,i}\circ\nabla\Phi_i)|_{F_{k,i}}(v_i)=L_{k,i}(v_i)G^0_{k,i}(v_i)$. Since $(\sigma_i^{-1})'(0)$ is twice the identity, the map $(s,v_i)\mapsto L_{k,i}(sv_i)G^0_{k,i}(v_i)$, $s\in[0,1]$, is a homotopy between $4G^0_{k,i}$ and $(P_{k,i}\circ\nabla\Phi_i)|_{F_{k,i}}$. It follows from Corollary \ref{cor:degree} that
\[
\deg(G_k^0, \cV\cap F_k,0) = (-1)^\ell
\] 
for any open bounded set $\cV\subset\sigma(\mathscr{U}^0)$ such that $0\in \cV$ and for all $k$ large enough. 

By Corollary \ref{cor:bound} we may choose $R>0$ such that all solutions $\bar u\in\mathscr{U}^t$ to \eqref{eq:constraint}-\eqref{eq:system} are contained in the open ball $B_R(0)$ of radius $R$ in $\cH$. Then $\sigma(B_R(0)\cap \mathscr{U}^t)$ is $H^1$-bounded in $T_{\bar\omega}(\mathscr T)$, uniformly in $t$. Let 
\[
D^t := B_R(0)\cap \mathscr{U}^t \quad \text{and} \quad \mathscr{U} := \{(u,t)\in \mathscr{T}\times[0,1]: u\in D^t\}.
\]
It is easy to see that $\mathscr{U}$ is open in $\mathscr{T}\times[0,1]$. We claim that there exists $k_0$ such that $P_kS^t(\bar u)\ne 0$ for all $\bar u\in \partial D^t\cap E_k$, $t\in[0,1]$ and $k\ge k_0$. Assuming the contrary, we find $t_n\to t$, $k_n\to\infty$ and $\bar u_n\in \partial D^{t_n}\cap E_{k_n}$ such that $P_{k_n}S^{t_n}(\bar u_n) = 0$. Since $\mathrm{dist}(\bar u_n, \partial D^{t})\to 0$, it follows from Lemma \ref{lem:convergence} that there exists $\bar u\in \partial D^{t}$ such that $S^{t}(\bar u)=0$. According to Lemma \ref{lem:positive} or Lemma \ref{lem:positive2} and the choice of $R$, this is impossible. Hence $k_0$ exists as claimed. 

Set $\cV_k^t := \sigma(D^t)\cap F_k$ and 
\[
\cV_k := \{(v,t)\in F_k\times[0,1]: v\in \cV_k^t\}.
\]
Then $G_k^t(\bar v)\ne 0$ for all $\bar v\in\partial\cV_k^t$ and $t\in[0,1]$. It follows from the extended homotopy invariance property of the degree (see e.g. \cite[Theorem 4.1]{am} or \cite[Theorem 2.1.1]{dm}) applied to $(G_k^t,t)$ on $\cV_k$ that the degree $\deg(G_k^t, \cV_k^t,0)$ is independent of $t$ whenever $k\ge k_0$. Hence,
\[
\deg(G_k^1, \cV_k^1,0) = \deg(G_k^0, \cV_k^0,0) = (-1)^\ell
\]
and, therefore, $G_k^1(\bar v_k)=0$ for some $\bar v_k\in \cV_k^1$ and every $k\ge k_0$. So setting $\bar u_k := \sigma^{-1}(\bar v_k)$, we have 
\[
P_kS^1(\bar u_k) = 0, \quad \bar u_k\in D^1\cap E_k,\quad k\ge k_0.
\]
Applying Lemma \ref{lem:convergence} again, we obtain $\bar u\in \ol{D^1}$ such that $S^1(\bar u)=0$  (a posteriori, $\bar u\in D^1$ by Lemma \ref{lem:positive} or Lemma \ref{lem:positive2} and the choice of $R$). It follows that $(\bar\kappa,\bar u)$ with $\bar\kappa:=\bar\kappa^1(\bar u)$ is a solution to \eqref{eq:constraint}-\eqref{eq:system} for $t=1$ and hence to \eqref{eq:mainsystem}-\eqref{eq:constraint}. This completes the proof.
\end{proof}

\appendix
\section{The covariant Hessian} \label{sec:hessian}

In this appendix we collect some facts on the relation between the Hessian of a functional on a Hilbert space and the covariant Hessian of a related function on a Hilbert submanifold, that play an important role in the proof of our main result.

Let $H$ be a Hilbert space, $M$ a smooth Hilbert submanifold of $H$ and $\mathscr X(M)$ be the space of $\cC^1$ tangent vector fields on $M$. Given $X,Y\in\mathscr X(M)$, the \emph{covariant derivative} of $Y$ in the direction of $X$ is the vector field $\nabla_XY\in\mathscr X(M)$ whose value at a point $u\in M$ is defined by
$$\nabla_XY|_u:=\pi_u(D\bar Y(u)[X(u)]),$$
where $\bar Y:H\to H$ is a local extension of $Y$ at $u$, $D\bar Y$ is the derivative of $\bar Y$, and $\pi_u:H\to T_uM$ is the orthogonal projection onto the tangent space of $M$ at $u$.

If $J:M\to\r$ is a $\cC^2$-function, its \emph{gradient} is the vector field $\nabla J\in\mathscr X(M)$ given by
$$
\langle \nabla J,X\rangle=XJ\qquad\text{for every \ }X\in\mathscr X(M),
$$
where $XJ$ is the derivative of $J$ in the direction of $X$. So, if $\bar J:H\to\r$ is a $\cC^2$ extension of $J$, then $\nabla J|_u=\pi_u(\nabla\bar J(u))$. The \emph{covariant Hessian} of $J$ is the $2$-tensor defined by
$$
\nabla^2J[X,Y]:=\langle \nabla_X\nabla J,Y\rangle\qquad\text{for every \ }X,Y\in\mathscr X(M).
$$

\begin{example} \label{example_hessian}
\emph{
If $M=H$ and $I:H\to\r$ is of class $\cC^2$, the gradient of $I$ is the usual gradient, the covariant Hessian at $u$ is given by
$$\nabla^2I[X,Y]\,|_u=\langle \mathbf{H}I(u)v,w\rangle=I''(u)[v,w], \qquad\text{where \ }v:=X(u), \ w:=Y(u),$$
and $\mathbf{H}I(u):=D(\nabla I)(u)$ is the linear mapping representing the usual Hessian because
\begin{align*}
\nabla^2I[X,Y]\,|_u&=\langle D(\nabla I)(u)[X(u)],Y(u)\rangle=\left\langle\lim_{t\to 0}\frac{\nabla I(u+tv)-\nabla I(u)}{t},w\right\rangle \\
&=\lim_{t\to 0}\frac{I'(u+tv)w-I'(u)w}{t}=I''(u)[v,w],
\end{align*}
see \cite[Remark 1.5]{w}.
}
\end{example}

$u\in M$ is a \emph{critical point} of $J$ if $\nabla J|_u=0$. It is called \emph{nondegenerate} if $\nabla^2J|_u:T_uM\times T_uM\to\r$ is a nondegenerate bilinear form.

\begin{lemma} \label{lem:hessian_critical}
If $u\in M$ is a critical point of $J$, the covariant Hessian at $u$ is given by
$$\nabla^2J[X,Y]\,|_u=X\langle\nabla J,Y\rangle|_u\qquad\text{for every \ }X,Y\in\mathscr X(M),$$
where $X\langle\nabla J,Y\rangle$ is the derivative of the function $\langle\nabla J,Y\rangle:M\to\r$ in the direction of $X$.
\end{lemma}

\begin{proof}
Since the connection on $M$ is Riemannian, we have that
$$X\langle\nabla J,Y\rangle=\langle\nabla_X\nabla J,Y\rangle + \langle\nabla J,\nabla_XY\rangle=\nabla^2J[X,Y] + \langle\nabla J,\nabla_XY\rangle.$$
If $u$ is a critical point of $J$, then $\langle\nabla J,\nabla_XY\rangle|_u=0$ and the claim follows.
\end{proof}

Let $G:H\to\r$ be a smooth function, $1$ be a regular value of $G$ and $M:=G^{-1}(1)$. Then $T_uM=\ker G'(u)$. Assume that $u\notin T_uM$ for all $u\in M$. 

As before, let $J:M\to\r$ be a $\cC^2$-function and $\bar J:H\to\r$ a $\cC^2$ extension of $J$. For $\lambda\in\r$ set $I:=\bar J+\lambda G$.

\begin{proposition} \label{prop:hessian}
Let $u\in M$. Then,
\begin{itemize}
\item[$(i)$] $u$ is a critical point of $I$ if and only if $u$ is a critical point of $J$ and 
$$\lambda=-\frac{\bar J'(u)u}{G'(u)u}.$$
\item[$(ii)$] If $u$ is a critical point of $I$, then 
$$\nabla^2J[X,Y]|_u=I''(u)[X(u),Y(u)]\qquad\text{for every \ }X,Y\in\mathscr X(M).$$
\end{itemize}
\end{proposition}

\begin{proof}
$(i):$ If $0=I'(u)v=[\bar J'(u)+\lambda G'(u)]v$ for all $v\in H$ then, as $G'(u)v=0$ for all $v\in T_uM$, we have 
\begin{equation*}
\langle \nabla J,X\rangle|_u=\bar J'(u)[X(u)]=I'(u)[X(u)]=0\qquad\text{for every \ }X\in\mathscr X(M)
\end{equation*}
and, as $I'(u)u=0$, we derive that
$$\lambda=-\frac{\bar J'(u)u}{G'(u)u},$$
as claimed. The converse is also easy.

$(ii):$ Since $G'(u)v=0$ for every $u\in M$ and $v\in T_uM$, we have that
$$\langle\nabla J,Y\rangle|_u=\langle\nabla\bar J(u),Y(u)\rangle=\langle\nabla I(u),Y(u)\rangle\qquad\text{for every \ }Y\in\mathscr X(M), \ u\in M,$$
i.e., the function $\langle\nabla I,\bar Y\rangle$ is an extension of $\langle\nabla J,Y\rangle$ to $H$ for any extension $\bar Y$ of $Y$. Hence, their directional derivatives satisfy
$$X\langle\nabla J,Y\rangle|_u=\bar X\langle\nabla I,\bar Y\rangle|_u\qquad \text{for every \ }X,Y\in\mathscr X(M), \ u\in M.$$
If $u$ is a critical point of $I$ then, by $(i)$, $u$ is a critical point of $J$. So the previous identity, together with Lemma \ref{lem:hessian_critical} and Example \ref{example_hessian}, yields
$$\nabla^2J[X,Y]\,|_u=\nabla^2I[X,Y]\,|_u=I''(u)[X(u),Y(u)]\qquad\text{for every \ }X,Y\in\mathscr X(M),$$
as claimed.
\end{proof}

\bigskip

\begin{flushleft}
\textbf{Mónica Clapp}\\
Instituto de Matemáticas\\
Universidad Nacional Autónoma de México\\
Circuito Exterior, Ciudad Universitaria\\
04510 Coyoacán, Ciudad de México, Mexico\\
\texttt{monica.clapp@im.unam.mx} 
\medskip

\textbf{Andrzej Szulkin}\\
Department of Mathematics\\
Stockholm University\\
106 91 Stockholm, Sweden\\
\texttt{andrzejs@math.su.se} 
\end{flushleft}

\end{document}